\documentclass[reqno]{amsart}
\usepackage{amssymb,amsmath}
\usepackage{amsthm}
\usepackage{color,graphicx}

\newcommand{\ps}{\psi(\xi,t)}

\newcommand{\lan}{\langle x \rangle}

\newcommand{\lanx}{\langle x \rangle}
\newcommand{\z}{\mathcal Z}

\newcommand{\R}{\mathbb R}

\newcommand{\D}{\mathcal{D}_\xi^\gamma}

\newcommand{\p}{\partial}
\newcommand{\e}{e^{-it\xi|\xi|-t|\xi|^{1+a}}}

\newcommand{\sgn}{\text{sgn}}

\newcommand{\ha}{\hat{\phi}}

\newcommand{\ste} {\mathcal{D}_{\xi}^\theta}
\newcommand{\ta}{t^{-\frac{a}{1+a}}}

\newcommand{\h} {\mathcal H}

\newtheorem{theorem}{Theorem}[section]
\newtheorem{proposition}[theorem]{Proposition}
\newtheorem{remark}[theorem]{Remark}
\newtheorem{lemma}[theorem]{Lemma}

\newtheorem{claim}[theorem]{Claim}

\begin{document}
\vglue-1cm \hskip1cm
\title[The dissipative Benjamin-Ono equation]{The Cauchy Problem for Dissipative Benjamin-Ono equation in Weighted Sobolev spaces}

\author[A. Cunha]{Alysson Cunha}
\address{IME-Universidade Federal de Goi\'as (UFG), 131, 74001-970, Goi\^an\-ia-GO, Bra\-zil}
\email{alysson@ufg.br}

\begin{abstract}

We study the well-posedness in weighted Sobolev spaces, for the initial value problem (IVP) associated with the dissipative Benjamin-Ono (dBO) equation. We establish persistence properties of the solution flow in the weighted Sobolev spaces $\z_{s,r}=H^s(\R)\cap L^2(|x|^{2r}dx)$, $s\geq r>0$. We also prove some unique continuation properties in these spaces. In particular, such results of unique continuation show that our results of well posedness are sharp.
\end{abstract}

\maketitle

\section{Introduction}\label{introduction}

This work is concerned with the initial-value problem (IVP) associated with the dissipative-Benjamin-Ono (dBO) equation
\begin{equation}\label{bodiss}
\begin{cases}
u_{t}+\mathcal{H}\partial_{x}^{2}u+D^{\alpha}u+uu_{x}=0, \;x\in\R, \;t>0, \\
u(x,0)=\phi(x),
\end{cases}
\end{equation}
where $0\leq \alpha\leq 2$, $D^s$ denotes the fractional derivative of order $s$ defined, via Fourier transform as $$D^sf(x)=(|\xi|^s\widehat{f})^\vee(x),$$
and $\h$ is the Hilbert transform defined by
\begin{equation}
\begin{split}
\h f(x)&=\mathrm{p.v.}\frac{1}{\pi}\int_{\R}\frac{f(y)}{x-y}dy\\
                         &=\mathcal F^{-1}(-i\sgn(\xi)\hat{f}(\xi))(x).
\end{split}
\end{equation}

We see that \eqref{bodiss} represents, for $\alpha=0$, the well-known Benjamin-Ono (BO) equation, that was deduced by Benjamin \cite{Benjamin} and later by Ono \cite{Ono} as a model for long internal gravity waves in deep stratified fluids. The BO equation has been extensively studied, in the last years, with respect to regularity in Sobolev spaces. In this sense, issues about locally or globally well-posedness (LWP and GWP, resp.) are addressed. In general, the main goal is to find the minimal regularity in Sobolev spaces, see \cite{BP}, \cite{tao} and \cite{IK}. Moreover, unique continuation principles are often investigated. For more details see, \cite{Iorio}, \cite{Iorio2} and \cite{ponce}.
Several papers on the study of local-well posedness or global-well posedness in weighted Sobolev has been devoted to BO equation, see \cite{Iorio1}, \cite{ponce}, \cite{FLP} and \cite{flores}. Solutions without infinity decay in the initial data can be found in \cite{iorio3} and \cite{fonseca}. Two-dimensional versions are also of great interest in the literature, see for example \cite{Aniura1}, \cite{APlow}, \cite{EP1}, \cite{AP}, \cite{bustamante}, \cite{EP3} and \cite{ribaud}. Recently, a higher order versions of the BO equation were studied, see \cite{riano}, \cite{riano2}, \cite{schippa}, and references therein. Kenig, Ponce and Vega \cite{kpv1} obtained results of uniqueness solutions for the BO equation, see also \cite{FLP}. 
 
For $\alpha=2$, the (dBO) becomes the Benjamin-Ono-Burgers (BOB) equation
\begin{equation}
\p_t u +\h \p_x^2 u-\p_x^2 u +uu_x=0.
\end{equation}
This equation was derived by Edwin and Roberts in \cite{Edwin}. 
In \cite{Otani1} Otani obtained global-well posedness in $H^s(\R)$, where $s>-1/2$. After that, Vento \cite{vento} showed this index is critical in the sense that the flow map $\phi \mapsto u$ is not of class $C^3$ from $H^s(\R)$ to $H^s(\R)$, for $s<-1/2$. 
By and large, issues such as well-posedness and asymptotic behavior of solutions, for the (BOB) equation, has been widely studied in the last years, see \cite{Bona}, \cite{Dix1}, \cite{Guo1}, \cite{Molinet1}, \cite{Zhang}, and references therein. 

The well-posedness for the dBO equation was first examined by Vento \cite{vento}. More precisely he obtained, in the case $1<\alpha\leq 2$, global well-posedness in $H^s(\R)$, where $s>-\alpha/4$, and Ill-posedness (holds when $\alpha=1$), for $s<-\alpha/4$, in the sense that the mapping data-solution is not $C^3$ in a neighborhood of the origin. If  $0\leq \alpha< 1$, he also obtained the Ill-posedness in $H^s(\R)$, for all $s\in \R$, in the sense that the mapping data-solution is not $C^2$ at origin. 
On the other hand, about the long-time behavior of the solutions, we can also add that the following perturbation of the IVP \eqref{bodiss}
\begin{equation}\label{bodissf}
\begin{cases}
u_{t}+\mathcal{H}\partial_{x}^{2}u+D^{\alpha}u+uu_{x}=f, \;x\in\R, \;t>0, \\
u(x,0)=\phi(x),
\end{cases}
\end{equation}
where $f\in L^2((1+x^2)^{1/2}dx)$, has a global attractor with finite dimension, in the sense of Hausdorff, see \cite{attAlarcon1}.

As the usual, we are considering the well-posedness in the Kato's sense, that is, includes, existence, uniqueness, persistence property and smoothness of the map data-solution. In this paper, we are mainly interested in to study the global well-posedness of the IVP \eqref{bodiss} in weighted Sobolev spaces.

Our main results are the following:

\begin{theorem} Let $a\in (0,1]$, then the following statements are true.

\begin{itemize}\label{well}
\item [i)] Let $s\geq r>0$ and $r<3/2+a$. Then, the IVP \eqref{bodiss} is GWP in $\mathcal{Z}_{s,r}$.
\item [ii)] Let $r\in [3/2+a,5/2+a)$ and $r\leq s$. Then, the IVP \eqref{bodiss} is GWP in $\dot{\mathcal{Z}}_{s,r}$.
\end{itemize}

\end{theorem}

The definition of spaces $\z_{s,r}$ and $\dot{\mathcal{Z}}_{s,r}$ can be found at the beginning of section 2. 

\begin{theorem}\label{P1}
Let $u\in C([0,T]; \mathcal{Z}_{1,1})$ be a solution of the IVP
\eqref{bodiss}, with $a\in (0,1]$. If there exist two different times $t_1, t_2 \in [0,T]$ such
that $u(t_j)\in \mathcal{Z}_{3/2+a,3/2+a}$, $j=1,2$, then $$\hat{u}(0,t)=0,
\ \quad \mbox{for\,\,all}\; \ t\in [0,T].$$
\end{theorem}

\begin{theorem}\label{P2}
Let $u\in C([0,T]; \mathcal{Z}_{2,2})$ be a solution of the $\mathrm{IVP}$
\eqref{bodiss}, with $a\in (0,1]$. If there exist three different times $t_1, t_2, t_3 \in [0,T]$ with $u(t_j)\in \mathcal{Z}_{5/2+a,5/2+a}$, $j=1,2,3$, then there exists $t_1<\bar{t}<t_2$ such that $$u(x,t)=0,
\ \quad \mbox{for\,\,all}\; \ x\in \R, \ t\geq \bar{t}.$$
\end{theorem}

Now we present some ingredients for the proof of Theorems \ref{well}--\ref{P2}. 

With respect to Theorems \ref{P1} and \ref{P2} we will adapt the techniques introduced by Fonseca, Linares and Ponce \cite{FLP1} for the study of the Benjamin-Ono equation with a generalized dispersion. It consists in the use of the Stein derivative for obtaining the unique continuation principles. Our proof is a little different from the one presented in \cite{FLP1}. In fact, we made use of an additional commutator estimate for the derivative $D_\xi^\gamma$ (see Proposition \ref{Comu} below), see also \cite{ponce} and \cite{AP}.

The proof of Theorem \ref{well} will be obtained by using the ideas of  Fonseca, Pastr\'an, and Rodr\'iguez-blanco \cite{pastran}, as well as the Stein-Weiss interpolation theorem with change of measure, see \cite{HBS}. In this work the authors studied a version of the BO equation with a dissipative effect. 
We note that the dissipative term in (dBO) equation, has the effect of pushing down the indices $r$ and $s$ of space $\z_{s,r}$, in relation those obtained in \cite{ponce} and \cite{FLP1}. 

Theorem \ref{P1} shows that the Theorem \ref{well}(part i)) is sharp in the sense that it's not possible to find an index $r$ more than $3/2+a$, so that the part i) still hold valid. Theorem \ref{P2} also shows that the Theorem \ref{well}(part ii)) is sharp, in the same sense previous. 

Reciprocally our well-posedness results show that the Theorems \ref{P1} and \ref{P2} are also sharp, in the sense that the indexes $3/2+a$ and $5/2+a$, respectively, cannot be pushed down.

Now we will describe some consequences of Theorems \ref{well}--\ref{P2}. By the Theorem \ref{well}(part i)) the decay $r=(3/2+a)-$ is optimal, so that the persistence property is satisfied for general initial data. More precisely, it shows that  the solution $u$ of the IVP \eqref{bodiss} with initial data $\phi \in \z_{s,r}$, $s\geq r\geq 3/2+a$, $\hat{\phi}(0)\not=0$, is such that 
$u\in C([0,T];\z_{s,(3/2+a)-})$, for $T>0$, but   
does not exist a non-trivial solution $u$ with initial data $\phi$ that verifies $u\in C([0,T];\z_{s,3/2+a})$.
   
The decay  $r=(5/2+a)-$ is the largest possible, in the following sense, by the Theorem \ref{well}(part ii)), for some $T>0$ there are non-trivial solutions $u\in C([0,T];\dot{\z}_{s,(5/2+a)-})$, but Theorem \ref{P2} implies that does not exist a non-trivial solutions $u\in C([0,T];\dot{\z}_{s,5/2+a})$. 

The rest of this paper is as follows. Section 2 contains some preliminary estimates that will be useful in the coming sections.
In the section 3 we prove the well-posedness. Theorem \ref{P1} will be proved in section 4. To finish, the proof of Theorem \ref{P2} will be present in section 5.

\section{Preliminaries}

\subsection{Notation}\label{notation}
In this paper, we use the following notation. We say $a \lesssim b$ if there exists a constant $c>0$ such that $a\leq c b$. We also write $a \lesssim_{l} b$  when the constant depends on only parameter $l$. The Fourier transform of $f$, is defined by
$$\hat{f}(\xi)=\int_{\R}e^{-i\xi x}f(x)dx.$$
If $s\in \R$, $H^s:=H^s(\R)$ represents the nonhomogeneous Sobolev space defined as
 $$H^s(\R)=\{f\in \mathcal{S}'(\R): \|f\|_{H^s}<\infty\},$$
 where
 $$\|f\|_{H^s}=\|\langle \xi \rangle^s \hat{f}\|_{L^2_\xi},$$
 and $\langle \xi \rangle=(1+\xi^2)^{1/2}$.
 In addition, we define the Bessel potential $J^s$ by
$$(J^s f)^{\wedge}(\xi)=\langle \xi \rangle^s\hat{f}(\xi), \ \mbox{for all} \ f\in \mathcal{S}'(\R),$$
hence $\|J^s f\|_{L^2_x}=\|f\|_{H^s}$. The weighted Sobolev space is defined by $$
\mathcal{Z}_{s,r}=H^s(\R)\cap L^2_r(\R)
,$$ where $L^2_r(\R)=L^2(\langle x\rangle^{2r}dx)$. The norm in $\mathcal{Z}_{s,r}$ is
given by
$\|\cdot\|_{\mathcal{Z}_{s,r}}^2=\|\cdot\|_{H^s}^2+\|\cdot\|_{L^2_r}^2$. We also introduced the notation
$$\dot{\mathcal{Z}}_{s,r}=\{f\in\mathcal{Z}_{s,r}: \hat{f}(0)=0\}.$$

For help in our estimates, we define the function $\chi\in C_0^\infty(\R)$, with $\mbox{supp}\chi\subset [-2,2]$ and $\chi\equiv1$ in $(-1,1)$.

In the rest of the paper, we will denote the $L^2$-norm in the $x$ variable by $\|\cdot\|_{L^2_x}:=\|\cdot\|$.

Next, we introduce some results which will be useful to
demonstrate our main results.

\begin{proposition}\label{Jota}
Let $\delta,\nu>0$ such that $J^{\delta}f\in L^{2}(\R)$ and
$\langle x\rangle^\nu f
\in L^{2}(\R).$ Then for any
$\beta \in (0,1)$
\begin{equation}\label{inter1}
\|J^{\beta \delta}(\langle x\rangle^{(1-\beta)\nu}f)\|\leq c\|\langle x
\rangle^\nu f\|^{1-\beta}\|J^{\delta}f\|^{\beta}.
\end{equation}
\end{proposition}
\begin{proof}
See \cite{ponce}. 
\end{proof}

\begin{remark}
 Assuming  that $u$ is sufficiently regular we obtain, for every t in which the solution there exists
\begin{equation}\label{consquan}
\int u(x,t)dx=\int \phi(x)dx.
\end{equation}
 This implies that
\begin{equation}\label{fourieru}
\hat{u}(0,t)=\hat{\phi}(0).
\end{equation}
Moreover
\begin{equation}\label{consdata}
\frac{d}{dt}\|u(t)\|^2<0,
\end{equation}

and
\begin{equation}\label{consdata1}
\frac{d}{dt}\int x u(x,t)dx=\frac12 \|u(t)\|^2.
\end{equation}

\end{remark}

By defining $L^{p}_{s}:=(1-\Delta)^{-s/2}L^{p}(\R^n)$, the following result characterizes these spaces.
\begin{theorem}\label{stein}
Let $b\in (0,1)$ and $2n/(n+2b)<p<\infty.$ Then $f\in L^{p}_{b}(\R^{n})$ if and only if
\begin{itemize}
\item [a)] $f\in L^{p}(\R^{n}),$ \item [b)]
$\mathcal{D}^{b}f(x)={\displaystyle \left (
\int_{\R^{n}}\frac{|f(x)-f(y)|^{2}}{|x-y|^{n+2b}}dy\right)^{1/2}} \in
L^{p}(\R^{n}),$ with,
\begin{equation}\label{equiv}
\|f\|_{b,p}\equiv \|(1-\Delta)^{b/2}f\|_{p}=\|J^{b}f\|_{p}\simeq \|f\|_{p}+\|D^{b}f\|_{p}\simeq \|f\|_{p}+\|\mathcal{D}^{b}f\|_{p},
\end{equation}
where for $s\in \R$
$$D^{s}=(-\Delta)^{s/2} \ with \ D^{s}=(\mathcal{H}\partial_{x})^{s}, \ \mbox{if} \ n=1.$$
\end{itemize}
\end{theorem}
\begin{proof}
See \cite{Stein}.
\end{proof}

From the previous theorem, part b), with $p=2$ and $b\in(0,1)$, we have
\begin{equation}\label{Leib}
\|\mathcal{D}^{b}(fg)\|_2 \leq \|f\mathcal{D}^{b}g\|_2 + \|g\mathcal{D}^{b}f\|_2.
\end{equation}

\begin{lemma}\label{Leibnitz}
Let $b\in (0,1)$ and $h$ a measurable function on $\R$ such that $h,h'\in L^{\infty}(\R)$. Then, for all $x\in \R$
\begin{equation}\label{Lei}
\mathcal{D}^b h(x)\lesssim \|h\|_{L^{\infty}}+\|h'\|_{L^\infty},
\end{equation}
moreover
\begin{equation}\label{Leibh}
\|\mathcal{D}^{b}(h f)\| \leq \|\mathcal {D}^b h\|_\infty \|f\| + \|h\|_\infty \|\mathcal{D}^{b}f\|.
\end{equation}
\end{lemma}
\begin{proof}
See \cite{pastran}.
\end{proof}


Now, we turn our attention to the (dBO) equation. The integral equation associated to the IVP \eqref{bodiss} is given by
\begin{equation}\label{inteq}
\begin{split}
u(t)=U(t)\phi-\int_0^t U(t-\tau)z(\tau)d\tau,
\end{split}
\end{equation}
where $z=\frac12\p_x u^2$, and the semigroup $U(t)$ is defined by
\begin{equation}\label{semigroup}
(U(t)\phi)^{\wedge}(\xi)=e^{-it\xi|\xi|-t|\xi|^\alpha} \ha, \ t\in [0,\infty).
\end{equation}
 Putting 
$\psi(\xi,t)=\e$, where $a\in (0,1]$, the following computations will be useful in our proofs.

\begin{equation}\label{Derksi}
\p_\xi (\psi(\xi,t)\ha)=\Big [-\Big(t(1+a)|\xi|^a \sgn(\xi)+2it|\xi|\Big)\ha+\p_\xi \ha\Big ]\psi(\xi,t),
\end{equation}

\begin{equation}\label{Derksi2}
\begin{split}
\p_\xi^2 (\psi(\xi,t)\ha)=&\Big [\Big(t^2 (1+a)^2|\xi|^{2a} -t(1+a)a|\xi|^{a-1}+4it^2 (1+a)|\xi|^{a+1}\sgn(\xi)\\
                         &- 4t^2 \xi^2 -2it \sgn(\xi) \Big)\ha -\Big( 2t(1+a)|\xi|^a \sgn(\xi)+4it|\xi|\Big)\p_\xi \ha\\
                         & +\p_\xi^2 \ha\Big]\psi(\xi,t),
\end{split}
\end{equation}
and 
\begin{equation}\label{Derksi3}
\begin{split}
\p_\xi^3 (\psi(\xi,t)\ha)=&\Big [\Big(3t^2 a(1+a)^2|\xi|^{2a-1}\sgn(\xi)-t(a^2-1)a|\xi|^{a-2}\sgn(\xi)-8t^2 \xi-4it\delta_\xi\\
                         &- (1+a)^3 t^3 |\xi|^{3a} \sgn(\xi) +2i(1+a)(4+3a) t^2 |\xi|^{a}\sgn(\xi)-\\
                         &-6i(1+a)^2t^3 |\xi|^{2a+1}-12(1+a)t^3 |\xi|^{a+2}\sgn(\xi)+8it^3|\xi|^3-\\
                         &-12t^2\xi\Big)\ha+\\
                         &+\Big(3(1+a)^2 t^2 |\xi|^{2a}+12i(1+a)t^2 |\xi|^{a+1} \sgn(\xi)-3a(1+a)t|\xi|^{a-1} -\\
                         &-12t^2 \xi^2-6it\sgn(\xi)\Big)\p_\xi \ha - \Big(3(1+a)t|\xi|^a \sgn(\xi)+6it|\xi|\Big)\p_\xi^2 \ha+\\
                         &+\p_\xi^3 \ha\Big]\psi(\xi,t).
\end{split}
\end{equation}
In the above $\delta_\xi$ stands the Dirac delta function with respect to $\xi$.

\begin{lemma}
Let $\lambda>0$, then
\begin{equation}\label{lambda}
\|\xi^{2\lambda}\psi(\xi,t)\|_{L^\infty_\xi}\leq c(a,\lambda)t^{-\frac{2\lambda}{a+1}},
\end{equation}
where $$c(a,\lambda)=\left(\frac{(a+1)e}{2\lambda}\right)^{-\frac{2\lambda}{a+1}}.$$
Moreover
\begin{equation}\label{lemasigma}
\||\xi|^\sigma e^{-t|\xi|^{1+a}}\|_{L^2_{\xi}}=c_{\sigma,a} t^{-\frac{2\sigma+1}{2(1+a)}}. 
\end{equation}
\end{lemma}
\begin{proof}
Since $|\xi^{2\lambda}\psi(\xi,t)|\leq \xi^{2\lambda} e^{-t|\xi|^{1+a}}$$:=\varphi(\xi,t)$, a simple computation gives us
$$\p_\xi \varphi(\xi,t)=\xi^{2\lambda}e^{-t|\xi|^{1+a}}(\frac{2\lambda}{\xi}-(1+a)t|\xi|^a).$$

Then $\p_\xi \varphi(\xi,t)=0$ if, and only if, $|\xi|=|\xi_0|=\big(\frac{2\lambda}{1+a}\big)^{\frac{1}{a+1}}t^{-\frac{1}{a+1}}$. In view of $\varphi(0,t)=0$ and $\varphi(\xi,t)\to 0$ with $\xi \to \infty$, we conclude that
$$|\xi^{2\lambda}\psi(\xi,t)|\leq \varphi(\xi_0,t)=c(a,\lambda)t^{-\frac{2\lambda}{a+1}}.$$

About identity \eqref{lemasigma}, by using the change of variables $\xi=t^{-1/1+a}w$

$$\||\xi|^\sigma e^{-t|\xi|^{1+a}}\|^2_{L^2_\xi}=t^\frac{-2\sigma-1}{1+a} \int w^{2\sigma}e^{-2|w|^{1+a}}dw=c_{\sigma,a}^2 t^\frac{-2\sigma-1}{1+a}.$$
\end{proof}

\begin{proposition}
For all $t>0$ and $\lambda\geq 0$, $U(t)\in \mathcal{B}(H^s(\R),H^{s+\lambda}(\R))$ and
\begin{equation}\label{reg}
\|U(t)\phi\|_{H^{s+\lambda}}\leq c(a,\lambda)(1+t^{-\frac{\lambda}{1+a}})\|\phi\|_{H^s}.
\end{equation}
\end{proposition}
\begin{proof}
Follows from \eqref{lambda}.
\end{proof}

\begin{proposition}
Let $u$ solution of IVP \eqref{bodiss}, then
\begin{equation}\label{reg2}
u\in C((0,T]; H^\infty), 
\end{equation}
and 
\begin{equation}\label{reg1} 
\|u(t)\|_{H^{s+\lambda}}\leq c(a,\lambda,T)t^{-\frac{\lambda}{1+a}}\|\phi\|_{H^s}, \ t\in (0,T], \ s>1/2, \ 0\leq \lambda<a+1.
\end{equation}
\end{proposition}
\begin{proof}
The proof of \eqref{reg2} follows by a well-known bootstrapping argument. We will prove only \eqref{reg1}. By using the integral equation \eqref{inteq}, \eqref{reg} and \cite[Lemma X4]{KP}
\begin{equation} 
\begin{split}
\|u(t)\|_{H^{s+\lambda}}\lesssim_{a,\lambda}&(1+t^{-\frac{\lambda}{1+a}})\|\phi\|_{H^s}+\int_0^t \big(1+(t-\tau)^{-\frac{1}{1+a}}\big)\|\p_x u(\tau)^2\|_{H^{s+\lambda-1}}d\tau\\
\lesssim_{a,\lambda}& (t^{\frac{\lambda}{1+a}}+1)t^{-\frac{\lambda}{1+a}}\|\phi\|_{H^s}+\\
&+\int_0^t \big((t-\tau)^{\frac{1}{1+a}}+1\big)(t-\tau)^{-\frac{1}{1+a}}\|u^2(\tau)\|_{H^{s+\lambda}} d\tau\\
\lesssim_{a,\lambda,T}&t^{-\frac{\lambda}{1+a}}\|\phi\|_{H^s}+\int_0^t (t-\tau)^{-\frac{1}{1+a}}\|u(\tau)\|_{L^\infty_x}\|u(\tau)\|_{H^{s+\lambda}} d\tau\\
\lesssim_{a,\lambda,T}&{t^{-\frac{\lambda}{1+a}}}\|\phi\|_{H^s}+\int_0^t (t-\tau)^{-\frac{1}{1+a}}\|u(\tau)\|_{H^s}\|u(\tau)\|_{H^{s+\lambda}} d\tau\\
\lesssim_{a,\lambda,T}&{t^{-\frac{\lambda}{1+a}}}\|\phi\|_{H^s}+M_T\int_0^t (t-\tau)^{-\frac{1}{1+a}}\|u(\tau)\|_{H^{s+\lambda}} d\tau,
\end{split}
\end{equation}
where $M_T=\sup_{[0,T]}\|u(t)\|_{H^s}$.
Then we conclude the proof by an application of a version of Gronwall's Lemma, see \cite[section 1.2.1]{henry}. 
\end{proof}
\begin{lemma}
 Let $a\in (0,1)$ and $\lambda\in \mathbb{Z}^{+}$ then
\begin{equation}
\|\mathcal{D}_\xi^{b}(\psi(\xi,t) |\xi|^\lambda \hat{f})\|\leq c_{a,\lambda} (t^{-\frac{\lambda}{1+a}}+t^{\frac{1-\lambda}{1+a}}+t^{\frac{a-\lambda}{1+a}})\|f\|+t^{-\frac{\lambda}{1+a}}\||x|^b f\|,\label{psi1}
\end{equation} 
moreover \eqref{psi1} still holds if $|\xi|^{\lambda}$ is exchanged by $|\xi|^{\lambda_1}\xi^{\lambda_2},\lambda=\lambda_1+\lambda_2, \lambda_1,\lambda_2\in \mathbb{Z}^{+}$.
\end{lemma}
\begin{proof}
First, we see that
\begin{equation}\label{dksi}
\p_\xi \psi=-\big(t(1+a)|\xi|^a \sgn(\xi)+2it|\xi|\big)\psi.
\end{equation}
Therefore using \eqref{Leib}, \eqref{lambda} and \eqref{Lei}
\begin{equation}
\begin{split}
\|\mathcal{D}^b_\xi (\psi|\xi|^\lambda \hat{f})\|&\lesssim \|\mathcal{D}^b_\xi (\psi |\xi|^\lambda)\hat f\|+\|\psi |\xi|^\lambda \mathcal{D}^b_\xi \hat f\|\\
&\lesssim \|\mathcal{D}^b_\xi (\psi |\xi|^\lambda)\|_\infty \|f\|+\|\psi |\xi|^\lambda\|_\infty\|\mathcal{D}^b_\xi \hat f\|\\
&\lesssim_b (\||\xi|^\lambda \psi\|_\infty + \|\p_\xi(|\xi|^\lambda \psi)\|_\infty)\|f\|+\||\xi|^\lambda \psi\|_\infty \|\mathcal{D}^b_\xi \hat f\|\\
&\lesssim_b (t^{-\frac{\lambda}{1+a}}+t^{\frac{1-\lambda}{1+a}}+t^{\frac{a-\lambda}{1+a}})\|f\|+t^{-\frac{\lambda}{1+a}}\||x|^b f\|.
\end{split}
\end{equation}
The proof is similar when $|\xi|^\lambda$ is substituted by $|\xi|^{\lambda_1} \xi^{\lambda_2}$.

\end{proof}



In our estimates we need of the following result about the Hilbert transform in weighted spaces.

\begin{lemma}\label{boundhilbert}
Let $-1/2<\nu<1/2$, then the Hilbert transform $\mathcal{H}$ is a bounded operator in $L^{2}(|x|^\nu dx),$ i.e.
\begin{eqnarray}\label{bound}
\|\mathcal{H}f|x|^\nu\| \lesssim \|f|x|^\nu \|.
\end{eqnarray}
\end{lemma}
\begin{proof}
See \cite{Hunt&Munk}, in which the more general version can be found.
\end{proof}



The next proposition is a key ingredient to obtain our estimates. It will be useful in the proof our main results. 
\begin{proposition}\label{Dstein}
For any $\theta \in (0,1)$ and $\gamma >0,$
$$\mathcal{D}^\theta (|\xi|^\gamma \chi (\xi))(\eta) \sim \left\{\begin{array}{lcc}
c|\eta|^{\gamma -\theta}+c_1,& \quad \gamma \not= \theta, |\eta|\ll 1, \\
c(-\ln |\eta|)^{1/2}, & \quad \gamma=\theta, |\eta|\ll 1,\\
\frac{c}{|\eta|^{1/2+\theta}}, & \quad  |\eta|\gg 1,
\end{array}\right.
$$
with $\mathcal{D}^\theta (|\xi|^\gamma \chi (\xi))(\cdot)$ continuous in $\eta \in \R-\{0\}.$ In particular, one has that
$$\mathcal{D}^\theta (|\xi|^\gamma \chi (\xi))(\cdot)\in L^{2}(\R) \ \mbox{if and only if} \ \theta< \gamma +1/2.$$
In a similar fashion
\begin{equation}\label{Dstein1}
\mathcal{D}^\theta (|\xi|^\gamma \sgn(\xi) \chi (\xi))(\cdot)\in L^{2}(\R) \ \mbox{if and only if} \ \theta< \gamma +1/2.
\end{equation}
\end{proposition}

\begin{proof}
See Proposition 2.9 in \cite{FLP1}.
\end{proof}

\begin{proposition}\label{DsteinL2}
Let $\gamma \in [0,1/2)$, then

\begin{equation}\label{Dstein2}
\mathcal{D}^\gamma (|\xi|^{\gamma-1/2}\chi (\xi))\notin L^{2}(\R).
\end{equation}
\end{proposition}
\begin{proof}
Let $\eta\in (0,1)$, then putting $\gamma_1 = \gamma-1/2$ and using a change of variables
\begin{equation}
\begin{split}
\mathcal{D}^\gamma (|\xi|^{\gamma_1}\chi(\xi))(\eta)^2 &=\int \frac{(|y|^{\gamma_1}\chi(y) -|\eta|^{\gamma_1}\chi(\eta))^{2}}{|y-\eta|^{1+2\gamma}}dy\\
&=\int \frac{(|\xi+\eta|^{\gamma_1} \chi(\xi+\eta)-|\eta|^{\gamma_1}\chi(\eta))^2}{|\xi|^{1+2\gamma}}d\xi\\
&= \int_{\xi \in (-\eta,0)}+\int_{\xi\notin (-\eta,0)}\\
&:= A(\eta)+B(\eta).
\end{split}
\end{equation}
In the first integral above we have $0<\xi+\eta<\eta<1$, hence $\chi(\xi+\eta)=\chi(\eta)=1$. By the Mean Value Theorem there exists $z \in (\xi+\eta,\eta)$ such that $$(\xi+\eta)^{\gamma_1} -\eta^{\gamma_1}=\gamma_1 z^{\gamma_1 -1} \xi.$$ Therefore
\begin{equation}
\begin{split}
A(\eta)
       &= \int_{-\eta}^0 \frac{((\xi+\eta)^{\gamma_1} -\eta^{\gamma_1})^2}{|\xi|^{1+2\gamma}}d\xi\\
       &=\gamma_1^2\int_{-\eta}^0 \frac{z^{2(\gamma_1 -1)}\xi^2}{|\xi|^{1+2\gamma}}d\xi\\
       &\geq \gamma_1^2 \int_{-\eta}^0 \frac{\eta^{2(\gamma_1 -1)}\xi^2}{|\xi|^{1+2\gamma}}d\xi\\
       &=\gamma_1^2 \eta^{2(\gamma_1 -1)} \int_{0}^\eta \xi^{1-2\gamma}d\xi\\
       &=\frac{\gamma_1^2}{2(1-\gamma)} \eta^{-1}\\
       &:=g(\eta).
\end{split}
\end{equation}
In view of $\mathcal{D}^\gamma (|\xi|^{\gamma_1}\chi(\xi))(\eta)^2\geq g(\eta)$ for any $0<\eta<1$ and $g\notin L^{1}_{loc}(\R)$, we conclude the proof.

\end{proof}

\begin{proposition}\label{Comu}
If $f\in L^{2}(\R)$ and $\phi \in H^{1}(\R),$ then
\begin{equation}
\|[D^{\gamma};\phi]f\|\leq c\|\phi\|_{H^2}\|f\|,
\end{equation}
where $\gamma \in (0,1)$.
\end{proposition}
\begin{proof}
We observe that
\begin{equation*}
\begin{split}
([D^{\gamma}; \phi]f)^{\wedge}(\xi)&=(D^{\gamma}(\phi f)-\phi D^{\gamma}f)^{\wedge}(\xi)\\
                                &=\int\big(|\xi|^{\gamma}-|\eta|^{\gamma}\big)\hat{\phi}(\xi-\eta)\hat{f}(\eta)d\eta.
\end{split}
\end{equation*}
It is easy to see that
$$||\xi|^{\gamma}-|\eta|^{\gamma}|\leq  |\xi-\eta|^{\gamma},$$
so
\begin{equation*}
\begin{split}
|([D^{\gamma}; \phi]f)^{\wedge}(\xi)|&\leq \int |\xi-\eta|^{\gamma}|\hat{\phi}(\xi-\eta)||\hat{f}(\eta)|d\eta\\
                                &=c(|\widehat{D^\gamma\phi}|\ast |\hat{f}|)(\xi).
\end{split}
\end{equation*}
Then, by the Young's inequality
\begin{equation*}
\begin{split}
\|[D^{\gamma}; \phi]f\|&\leq c \||\widehat{D^{\gamma}\phi}|\ast |\hat{f}|\| \\
                                   &\leq c \|\widehat{D^{\gamma}\phi}\|_{L^{1}}\|\hat{f}\|\\
																	 &\leq c \|\phi\|_{H^{2}}\|f\|,
\end{split}
\end{equation*}

where above we use that $$\|\widehat{D^{\gamma}\phi}\|_{L^{1}}\leq \|D^{\gamma}\phi\|_{H^1}\leq \|\phi\|_{H^2}.$$ This finishes the proof.
\end{proof}

\section{Well-posedness}

In the following we will obtain the global well-posedness. First we note that the case $a=1$ can be approach by using the ideas in \cite{pastran}. Thus, we deal only with the case $a\in (0,1)$.
First, we need of the next result.

\begin{lemma}
Let $\theta\in (0,1]$. Then for all $\phi \in \z_{\theta,\theta}(\R)$ 
\begin{equation}\label{well1}
\|D_\xi^\theta (\psi(\xi,t)\ha)\|\lesssim_a \rho_1(t)\|\lanx^\theta \phi\|,
\end{equation}
where $\rho_1(t)=3+t^{\frac{1}{1+a}}+t^{\frac{a}{1+a}}$.
\end{lemma}
\begin{proof}
Let $\theta \in (0,1)$, then by an application of the Parseval identity, \eqref{equiv} and \eqref{psi1}(with $\lambda=0$), we obtain
 \begin{equation}
 \begin{split}\label{initial}
\|D_\xi^\theta (\psi(\xi,t)\ha)\|&\lesssim \|\psi(\xi,t) \ha\|+\|\ste (\psi(\xi,t)\ha)\|\\
&\lesssim_a \|\phi\|+(1+t^{\frac{1}{1+a}}+t^{\frac{a}{1+a}})\|\phi\|+\||x|^\theta \phi\|\\
                         &\lesssim_a \rho_1(t)\|\lanx^\theta \phi\|.
 \end{split}
\end{equation}

The case $\theta=1$ follows by similar way.
\end{proof}

\begin{lemma}
Let $\theta\in (1/2,1/2+a)$. Then for all $\phi\in\z_{1+\theta,1+\theta}(\R)$
\begin{equation}\label{well2}
\|D_\xi^{\theta+1} (\psi(\xi,t)\ha)\|\lesssim_a \rho_2(t)\|\lanx^{\theta+1} \phi\|,
\end{equation}

where $\rho_2$ is a continuous increasing function on $[0,\infty)$. 
\end{lemma}
\begin{proof}
By using \eqref{psi1} and \eqref{dksi}
 \begin{equation}
 \begin{split}
\|D^{1+\theta}_\xi (\psi(t,\xi)\ha)\|\leq& \ \|D^{\theta}_\xi (\psi \p_\xi \ha)\|+\|D^{\theta}_\xi (\p_\xi \psi \ha)\|\\
                          \lesssim_a & (1+t^{\frac{a}{1+a}}+t^{\frac{1}{1+a}})\|\p_\xi \ha\|+\|D^\theta_\xi \p_\xi \ha\|+\\
                          &+\underbrace{t\|D^{\theta}_\xi(|\xi|^a \sgn(\xi)\psi \ha)\|}_{A_1}+\underbrace{t\|D_\xi^\theta(|\xi|\psi \ha)\|}_{A_2}.\\
 \end{split}
\end{equation}
We also have, from \eqref{psi1} 

\begin{equation}
\begin{split}
A_2&\lesssim 
   (t+t^\frac{a}{1+a}+t^\frac{2a}{1+a})\|\phi\|+t^{\frac{a}{1+a}}\||x|^\theta \phi\|,
\end{split}
\end{equation}
and
\begin{equation}
\begin{split}
A_1\lesssim& t\|D^\theta_\xi(|\xi|^a \sgn(\xi)\psi \ha \chi)\|+t\|D^\theta_\xi(|\xi|^a \sgn(\xi)\psi \ha (1-\chi))\|\\
   &:=A_{1,1}+A_{1,2},
\end{split}
\end{equation}
where $\chi$ is defined as before.

On the other hand, by \eqref{equiv} and \eqref{Leibh}

\begin{equation}
\begin{split}
A_{1,2}\lesssim& t (\||\xi|^a \psi(1-\chi)\|_\infty +\|\p_\xi (|\xi|^a \sgn(\xi) \psi(1-\chi))\|)\|\phi\|+\\
&+t\||\xi|^a \psi(1-\chi)\|_\infty \|\ste \ha\|\\
&\lesssim (2+t^{\frac{1-a}{1+a}}+t^{\frac{1}{1+a}})\|\phi\|+t^{\frac{1}{1+a}}\||x|^\theta \phi\|.
\end{split}
\end{equation}

We also can write

\begin{equation}
\begin{split}\label{A11}
A_{1,1}\lesssim& \ t\|D^\theta_\xi (|\xi|^a \sgn(\xi)(\psi-1)\ha \chi)\|+t\|D^\theta_\xi (|\xi|^a \sgn(\xi)\ha \chi)\|\\
       &:=A_{1,1}^1+A_{1,2}^2,
\end{split}
\end{equation}
 again by \eqref{Leibh} and \eqref{dksi}

\begin{equation}
\begin{split}
A_{1,1}^1 \lesssim&\  t(\||\xi|^a (\psi-1)\chi\|_\infty +\|\p_\xi (|\xi|^a \sgn(\xi) (\psi-1)\chi)\|_\infty)\|\phi\|+\\
 &+t\||\xi|^a (\psi-1)\chi\|_\infty \|\ste \ha\|)\\
 \lesssim_a & \ (t+t^2) (\|\phi\|+\||x|^\theta \phi\|).
\end{split}
\end{equation}
The last term in \eqref{A11} can be decomposed as
\begin{equation}
\begin{split}
A_{1,2}^2\lesssim& t \|D^\theta_\xi (\underbrace{|\xi|^a \sgn(\xi)(\ha(\xi)-\ha(0))\chi}_{L})\|+t\|D^\theta_\xi (|\xi|^a \sgn(\xi)\ha(0)\chi)\|\\
         &:=\tilde{A}+\tilde{\tilde {A}}.
\end{split}
\end{equation}

Thus
\begin{equation}
\begin{split}\label{des4}
\|L\|&\lesssim \||\xi|^a \ha(\xi)\chi\|+\||\xi|^a \ha(0)\chi\|\\
&\lesssim  \||\xi|^a \chi\|\|\ha\|_\infty\\
&\lesssim_a  \|J_\xi \ha\|\\
&= \|\lanx \phi\|,
\end{split}
\end{equation}
and in view of $\theta>1/2$
\begin{equation}
\begin{split}\label{des5}
\|\p_\xi L\|\lesssim& \ \||\xi|^a (\frac{\ha(\xi)-\ha(0)}{\xi})\chi\|+\||\xi|^a \p_\xi \ha(\xi)\chi\|+\||\xi|^a (\ha(\xi)-\ha(0))\p_\xi \chi\|\\
     \lesssim& \ \||\xi|^a \chi\|\|\p_\xi \ha\|_\infty +\||\xi|^a \chi\|_\infty \|\p_\xi \ha\|+\||\xi|^a \p_\xi \chi\|\|\ha\|_\infty+\\
     &+\||\xi|^a \p_\xi \chi\||\ha(0)|\\
     \lesssim_a& \ \|J^\theta \p_\xi \ha\|+\|x\phi\|+\|J^\theta \ha\|\\
     \lesssim_a& \ \|\lanx^{1+\theta}\phi\|.
\end{split}
\end{equation}
From \eqref{des4} and \eqref{des5} we obtain $$\tilde{A}\lesssim t\|L\|_{H^1_\xi}\lesssim t\|\lanx^{1+\theta}\phi\|.$$

Since $\theta<1/2+a$, by \eqref{Dstein1} we obtain $\ste(|\xi|^a \sgn(\xi)\chi) \in L^2$. Hence, using \eqref{equiv}
\begin{equation}
\tilde{\tilde {A}}\lesssim_a t\|\ha\|_\infty+ t\|\ha\|_\infty \|\ste(|\xi|^a \sgn(\xi)\chi)\| \lesssim_a t\|\lanx \phi\|.
\end{equation}
Gathering the inequalities above we conclude the proof.
\end{proof}
\begin{lemma}
Let $\theta\in (1/2,1/2+a)$. Then for all $\phi\in\dot{\z}_{2+\theta,2+\theta}(\R)$
\begin{equation}\label{well3}
\|D_\xi^{\theta+2} (\psi(\xi,t)\ha)\|\lesssim_a \rho_3(t)\|\lanx^{\theta+2} \phi\|,
\end{equation}

where $\rho_3$ is a continuous increasing function on $[0,\infty)$. 

\end{lemma}
\begin{proof}
Let $\phi\in \dot{\z}_{2+\theta,2+\theta}$, then by Plancherel identity

\begin{equation}
\begin{split}\label{ineq}
\|D_\xi^\theta \p_\xi^2 (\psi(\xi,t)\ha)\|\lesssim& \ \|D_\xi^\theta (\p_\xi^2 \psi(\xi,t)\ha)\|+\|D_\xi^\theta (\p_\xi \psi \p_\xi \ha)\|+\|D_\xi^\theta (\psi \p_\xi^2 \ha)\|\\
                          &:=C+D+E,
\end{split}
\end{equation}
thus by \eqref{psi1}
\begin{equation}
\begin{split}
E\lesssim& \ (1+t^{\frac{a}{1+a}}+t^{\frac{1}{1+a}})\|\p_\xi^2 \ha\|+\|\ste \p_\xi^2 \ha\|\\
\lesssim& \ (2+t^{\frac{a}{1+a}}+t^{\frac{1}{1+a}})\|x^2 \phi\|+\||x|^{\theta+2} \phi\|.
\end{split}
\end{equation}
Using \eqref{dksi}
\begin{equation}
\begin{split}
D&\lesssim t\|D_\xi^\theta(|\xi|^a \sgn(\xi)\psi \p_\xi \ha)\|+t\|D_\xi^\theta (|\xi| \psi \p_\xi \ha)\|:=D_1+D_2,
\end{split}
\end{equation}

where
\begin{equation}
\begin{split}
D_1 \lesssim& \ t\|D_\xi^\theta(|\xi|^a \sgn(\xi)\psi \chi \p_\xi \ha)\|+t\|D_\xi^\theta (|\xi|^a \sgn(\xi)\psi(1-\chi)\p_\xi \ha)\|\\
&:=D_{1,1}+D_{1,2}.
\end{split}
\end{equation}
We also can write
\begin{equation}
\begin{split}
D_{1,1}\lesssim & \ t(\|D_\xi^\theta (|\xi|^a \sgn(\xi) \psi \chi (\p_\xi \ha(\xi)-\p_\xi \ha(0)))\|+ \|D_\xi^\theta (|\xi|^a \sgn(\xi)\psi \chi \p_\xi \ha(0))\|)\\
 &:=D_{1,1}^1+D_{1,1}^2,
\end{split}
\end{equation}

hence by \eqref{equiv} and \eqref{Leibh}
\begin{equation}
\begin{split}
D_{1,1}^1 \lesssim& \ t\Big(\||\xi|^a \psi  (\p_\xi \ha(\xi)-\p_\xi \ha(0)))\|_\infty+\||\xi|^a \psi (\frac{\p_\xi \ha(\xi)-\p_\xi \ha(0)}{\xi})\|_\infty\\
&+\||\xi|^a \p_\xi \psi (\p_\xi \ha(\xi)-\p_\xi \ha(0))\|_\infty+\||\xi|^a \psi \p_\xi^2 \ha\|_\infty \Big) \|\chi\|\\
&+t\||\xi|^a \psi (\p_\xi \ha(\xi)-\p_\xi \ha(0))\|_\infty \|\ste \chi\|\\
\lesssim_a& \  \Big((t^{\frac{1}{1+a}} +t^{\frac{2}{1+a}}+t)\|\p_\xi \ha\|_\infty+t^{\frac{1}{1+a}}\|\p_\xi^2 \ha\|_\infty\Big)\|\chi\|+\\
&+t^{\frac{1}{1+a}}\|\p_\xi \ha\|_\infty\|\ste \chi\|\\
\lesssim_a& \ \Big((t^{\frac{1}{1+a}} +t^{\frac{2}{1+a}}+t)\|\lanx^{\theta+1}\phi\|+t^{\frac{1}{1+a}}\|\lanx^{\theta+2}\phi\|\Big)\|\chi\|+\\
&+t^{\frac{1}{1+a}}\|\lanx^{\theta+1}\phi\|\|\ste \chi\|,
\end{split}
\end{equation}

and 
\begin{equation}
\begin{split}\label{D11N}
D_{1,1}^2 \lesssim& \ t \Big(\|\psi\|_\infty+\|\p_\xi \psi\|_\infty)\||\xi|^a \p_\xi \ha(0)\chi\|+\\
&+|\p_\xi \ha(0)|\|\ste (|\xi|^a \sgn(\xi) \chi)\|\Big)\\
\lesssim& \ t \|\p_\xi \ha\|_\infty \Big(\|\psi\|_\infty+\|\p_\xi \psi\|_\infty)\||\xi|^a \chi\|+\|\ste (|\xi|^a \sgn(\xi) \chi)\|\Big)\\
\lesssim_a& \ \|\lanx^{\theta+1}\phi\| \Big((t+t^{\frac{2+a}{1+a}}+t^{\frac{1+2a}{1+a}})+t\underbrace{\|\ste (|\xi|^a \sgn(\xi) \chi)\|}_{N}\Big).
\end{split}
\end{equation}
From \eqref{Dstein1} follows that $N \in L^2$. On the other hand by \eqref{Leibh} and \eqref{psi1}
\begin{equation}
\begin{split}
D_{1,2}\lesssim& \ t\Big((\||\xi|^a (1-\chi)\psi\|_\infty+\|\p_\xi(|\xi|^a(1-\chi)\psi)\|_\infty)\|\p_\xi \ha\| + \\
&+\||\xi|^a (1-\chi)\psi\|_\infty\|\ste \p_\xi \ha\|\Big)\\
\lesssim& \ t\Big((\ta+t^{-\frac{2a}{1+a}}+t^{-1})\|\p_\xi \ha\|+\ta\|\ste \p_\xi \ha\|\Big)\\
\lesssim& \ (t^{\frac{1}{1+a}}+t^\frac{1-a}{1+a}+1)\|x\phi\|+t^\frac{1}{1+a}\|\lanx^{1+\theta}\phi\|,
\end{split}
\end{equation}
also from \eqref{equiv} and \eqref{psi1}
\begin{equation}
\begin{split}
D_2 
\lesssim& (t^\frac{a}{1+a}+t+t^\frac{2a}{1+a})\|x\phi\|+t^\frac{1}{1+a}\|\lanx^{1+\theta}\phi\|.
\end{split}
\end{equation}
Using \eqref{Derksi2}, the first term on the right-hand side of \eqref{ineq} can be estimated as
\begin{equation}
\begin{split}
C\lesssim& \ t^2 \|D_\xi^\theta(|\xi|^{2a}\psi \ha)\|+t\|D_\xi^\theta(|\xi|^{a-1}\psi \ha)\|+t^2 \|D_\xi^\theta(|\xi|^{a+1}\sgn(\xi)\psi \ha)\|+\\
&+t^2\|D_\xi^\theta(\xi^2 \psi \ha)\|+t\|D_\xi^\theta(\sgn(\xi)\psi \ha)\|\\
&:=C_1+C_2+C_3+C_4+C_5,
\end{split}
\end{equation}
where by \eqref{psi1}
\begin{equation}
\begin{split}
C_1
   \lesssim& \ (t^\frac{2}{1+a}+t^{\frac{3}{1+a}}+t^{\frac{2+a}{1+a}})\|\phi\|+t^\frac{2}{1+a}\||x|^\theta \phi\|,
\end{split}
\end{equation}

and

\begin{equation}
\begin{split}
C_2\lesssim& \ t(\|D_\xi^\theta (|\xi|^{a-1}\psi \chi \ha)\|+\|D_\xi^\theta(|\xi|^{a-1}\psi(1-\chi)\ha)\|)\\
\lesssim& :=C_2^1+C_2^2.
\end{split}
\end{equation}
The term $t|\xi|^{a-1}\psi(1-\chi)\ha\in H^1_\xi$, in fact
\begin{equation}
\begin{split}\label{H}
t\||\xi|^{a-1}\psi(1-\chi)\ha\|\lesssim& \ t\||\xi|^a \psi \ha \frac{1-\chi}{\xi}\|\\
\lesssim& \ t \|\frac{1-\chi}{\xi}\|_\infty\|\phi\|t^{-\frac{a}{1+a}}\\
\lesssim& \ t^{\frac{1}{1+a}}\|\phi\|,
\end{split}
\end{equation}

and
\begin{equation}
\begin{split}\label{H1}
t\|\p_\xi(|\xi|^{a-1}\psi(1-\chi)\ha)\|\lesssim& \ t\Big(\||\xi|^a \psi \ha \frac{1-\chi}{\xi^2}\|+t\||\xi|^{2a}\psi \ha \frac{1-\chi}{\xi}\|+\\
&+t\||\xi|^{a+1} \psi \ha \frac{1-\chi}{\xi}\|+\||\xi|^a \psi \p_\xi \ha\frac{1-\chi}{\xi}\|+\\
&+\||\xi|^a \psi \ha \p_\xi(\frac{1-\chi}{\xi})\|\Big)\\
\lesssim& \ (t^{\frac{1}{1+a}}+t^{\frac{1-a}{1+a}}+t^{\frac{a}{1+a}})\|\phi\|+t^{\frac{1}{1+a}}\|x\phi\|.
\end{split}
\end{equation}

For the estimate of $C_2^1$ we use that $\ha(0)=0$. An application of Taylor's formula
\begin{equation}
\ha(\xi)=\xi \p_\xi \ha(0) +\int_0^\sigma (\xi-\sigma)\p_\xi^2\ha(\sigma)d\sigma,
\end{equation}

leads us to 
\begin{equation}
\begin{split}\label{taylor1}
C_2^1&=tD_\xi^\theta(|\xi|^{a-1}\xi \psi \chi \p_\xi\ha(0))+tD_\xi^\theta\Big(\underbrace{\int_0^\xi (\xi-\tau) |\xi|^{a-1}\psi \chi\p_\xi^2\ha(\sigma)d\sigma}_{R}\Big)\\
&:=S+tD_\xi^\theta R,
\end{split}
\end{equation}
where $R\in H^1_\xi$. In fact, using Sobolev embedding

\begin{equation}
\begin{split}
\|R\|\lesssim& \ \||\xi|^{a-1} \psi \chi \xi^2 \|\p_\xi^2 \ha\|_\infty\|\\
     \lesssim& \ \|\p_\xi^2 \ha\|_\infty \||\xi|^{a+1}\psi\chi\|\\
     \lesssim_a& \  \|\lanx^{2+\theta}\phi\|,
\end{split}
\end{equation}

and
\begin{equation}
\begin{split}
\|\p_\xi R\|\lesssim_a& \ \||\xi|^{a-2}\psi \chi \xi^2 \|\p_\xi^2 \ha\|_\infty\|+t\|(|\xi|^{2a-1}+|\xi|^{a})\psi\chi\xi^2 \|\p_\xi^2 \ha\|_\infty \|+\\
&+\||\xi|^{a-1}\psi \chi' \xi^2 \|\p_\xi^2 \ha\|_\infty\|+\||\xi|^{a-1} \psi \chi \int_0^\xi \p_\xi^2 \ha(\tau)d\tau\|\\
\lesssim_a& \ \|\lanx^{2+\theta}\phi\|.
\end{split}
\end{equation}
From \eqref{equiv} and \eqref{Leibh}, the first term on the right-hand side of \eqref{taylor1} can be estimated as 
\begin{equation}
\begin{split}
\|S\|&\lesssim t (\|\psi\|_\infty+\|\p_\xi \psi\|_\infty)|\p_\xi \ha(0)|\||\xi|^a 
\chi\|+|\p_\xi \ha(0)|\|\ste (|\xi|^a \sgn(\xi)\chi)\|\\
&\lesssim_a \|\lanx^{1+\theta}\phi\|\big(t+t^{\frac{2+a}{1+a}}+t^{\frac{1+2a}{1+a}}+N\big),
\end{split}
\end{equation}
where $N$ is given by \eqref{D11N}.

Also by \eqref{Leibh} 
\begin{equation}
\begin{split}
C_3\lesssim& \ t^2 \Big((\||\xi|^{a+1}\psi\|_\infty+\|\p_\xi(|\xi|^{a+1}\psi)\|_\infty)\|\phi\|+\||\xi|^{a+1}\psi\|_\infty\|\ste \ha\|\Big)\\
\lesssim & \ (t+t^\frac{2+a}{1+a}+t^\frac{1+2a}{1+a})\|\phi\|+t\||x|^\theta \phi\|.
\end{split}
\end{equation}
Inequalities \eqref{equiv} and \eqref{psi1} implies that
\begin{equation}
\begin{split}
C_4=t^2 \|D_\xi^\theta(\xi^2 \psi \ha)\|
\lesssim& \ (t^\frac{2a}{1+a}+t^\frac{2a+1}{1+a}+t^\frac{3a}{1+a})\|\phi\|+t^\frac{2a}{1+a}\||x|^\theta \phi\|,
\end{split}
\end{equation}

and 

\begin{equation}
\begin{split}\label{C5}
C_5=t\|D_\xi^\theta(\sgn(\xi)\psi \ha)\|\lesssim& \ t\Big((1+t^\frac{1}{1+a}+t^\frac{a}{1+a})\|\phi\|+\|\ste(\sgn(\xi)\ha)\|\Big).
\end{split}
\end{equation}
We can deal with the last term in \eqref{C5} as follows. Since that $$-1/2<\theta-1<1/2,$$ by Lemma \ref{boundhilbert} 

\begin{equation}
\begin{split}
\|D_\xi^\theta(\sgn(\xi)\ha)\|=&\||x|^\theta \h \phi\|\\
=&\||x|^{\theta-1}x\h \phi\|\\
=&\||x|^{\theta-1}\h(x \phi)\| \ \\
\lesssim& \||x|^{\theta-1}x\phi\|\\
=&\||x|^\theta \phi\|.
\end{split}
\end{equation}
Hence by \eqref{equiv} $$\|\ste(\sgn(\xi)\ha)\|\lesssim \|\phi\|+\|D_\xi^\theta(\sgn(\xi)\ha)\|\lesssim \|\lanx^{1+\theta}\phi\|.$$
Then we conclude the proof.
\end{proof}

\begin{proof}[Proof of Theorem \ref{well}]

Let $s\geq r$ and $\phi \in \z_{s,r}(\R)$, then as we already observed, the solution $u(t)$ of \eqref{bodiss} is unique and satisfies $u\in C([0,T];H^s(\R))$, for all $T>0$. Also we have the continuous dependence on the initial data in $H^s(\R)$. Thus, in the following we will prove the persistence property in $L^2_r$. By putting $M:=\sup_{[0,T]}\|u(t)\|_{H^s}$ we will divide in several cases.
 
 Case 1). $r=\theta$, $\theta \in (1/2,1)$. We recall that $z=\frac12\p_x u^2$, then by using the integral equation \eqref{inteq}, for all $t\in [0,T]$
 \begin{equation}\label{int3}
\||x|^\theta u(t)\|\leq \||x|^\theta U(t)\phi\|+\int_0^t \||x|^\theta U(t-\tau)z(\tau)\|d\tau.
 \end{equation}
Therefore by Plancherel identity and \eqref{well1}
\begin{equation}\label{des8}
\begin{split}
\||x|^\theta U(t)\phi\|&=\|D_\xi^\theta (\psi(\xi,t)\ha)\|\\
&\lesssim_a \rho_1(t)\|\lanx^\theta \phi\|\\
                       &\lesssim_a \rho_1(T)\|\lanx^\theta \phi\|.
\end{split}
\end{equation}
With respect to integral term, \eqref{well1}, Sobolev embedding and \eqref{reg1} implies that   
\begin{equation}
\begin{split}\label{des9}
\||x|^\theta U(t-\tau)z(\tau)\|&=\|D_\xi^\theta (\psi(\xi,t-\tau)\hat{z}(\tau))\|\\
&\lesssim_a \rho_1(t-\tau)\|\lanx^\theta \p_x u^2(\tau)\|\\
&\lesssim_a \rho_1(T)\|\p_x u(\tau)\|_{L^\infty_x}\|\lanx^\theta u(\tau)\|\\
&\lesssim_a \rho_1(T)\|u(\tau)\|_{H^{s+\lambda}}\|\lanx^\theta u(\tau)\|\\
&\lesssim_a \rho_1(T)c(a,\lambda,T)t^{-\frac{\lambda}{1+a}}\|\phi\|_{H^s}\|\lanx^\theta u(\tau)\|,
\end{split}
\end{equation} 
where $3/2-s<\lambda<a+1$. 
 
In view of \eqref{consdata} 
\begin{equation}\label{pers2}
\|u(t)\|\leq \|\phi\|.
\end{equation}
Then using \eqref{des8}--\eqref{pers2} and Gronwall's lemma, see \cite[Lemma 7.1.1]{henry}, we conclude
\begin{equation}\label{pers1}
\|\lanx^\theta u(t)\|\lesssim_{a,T} \|\lanx^\theta \phi\| \ t\in [0,T].
\end{equation}
 
By \eqref{pers1} we obtain the persistence property in $L^2_r$. The continuity of application $t\in [0,T]\mapsto L^2_r$ follows by using \eqref{pers1}, see \cite{AP}. Using similar arguments to \cite{ponce} and \cite{AP} we can show the continuous dependence in $L^2_r$.


Case 2). $r=1+\theta$, $\theta \in (1/2, 1/2+a)$. By using \eqref{well2}, for all $t\in [0,T]$

\begin{equation}\label{des10}
\begin{split}
\||x|^{1+\theta} U(t)\phi\|&=\|D_\xi^{1+\theta} (\psi(\xi,t)\ha)\|\\
&\lesssim_a \rho_2(t)\|\lanx^{1+\theta} \phi\|\\
                       &\lesssim_a \rho_2(T)\|\lanx^{1+\theta} \phi\|,
\end{split}
\end{equation}
and from the integral term in \eqref{inteq}
\begin{equation}
\begin{split}\label{des9}
\||x|^{1+\theta} U(t-\tau)z(\tau)\|&=\|D_\xi^{1+\theta} (\psi(\xi,t-\tau)\hat{z}(\tau))\|\\
&\lesssim_a \rho_2(t-\tau)\|\lanx^{1+\theta} \p_x u^2(\tau)\|\\
&\lesssim_a \rho_2(T)\|\p_x u(\tau)\|_{L^\infty_x}\|\lanx^{1+\theta} u(\tau)\|\\
&\lesssim_a \rho_2(T)\|u(\tau)\|_{H^s}\|\lanx^{1+\theta} u(\tau)\|\\
&\lesssim_a \rho_2(T)M\|\lanx^{1+\theta} u(\tau)\|,
\end{split}
\end{equation} 
where above we use the Sobolev embedding.
From now we can proceed as in the last case to conclude the result.

Case 3). $1<r<3/2$. From inequality \eqref{well1}

\begin{equation}
\|U(t)\phi\|_{L^2_1}\lesssim_a \rho_1(t)\|\phi\|_{L^2_1},
\end{equation}
and by \eqref{well2}
\begin{equation}
\|U(t)\phi\|_{L^2_\sigma}\lesssim_a \rho_2(t)\|\phi\|_{L^2_\sigma}, \ \sigma>3/2.
\end{equation}
Then by the Stein-Weiss interpolation theorem with change of measures (see \cite{HBS})
\begin{equation}\label{interp}
\|U(t)\phi\|_{L^2_r}\lesssim_a \rho(t)\|\phi\|_{L^2_r}\lesssim_{a,T}\|\phi\|_{L^2_r}, \ t\in [0,T],
\end{equation}
where $r=1+\theta(\sigma-1)$ and $\rho(t)\leq \rho_1(t)^{1-\theta}\rho_2(t)^{\theta}$, $\theta\in (0,1)$. 

Therefore by \eqref{inteq}, \eqref{interp}, Sobolev embedding and \eqref{reg1}    
\begin{equation}
\begin{split}
\|u(t)\|_{L^2_r}\lesssim_{a,T}&\|\phi\|_{L^2_r}+\frac12\int_0^t \|\p_x u(\tau)^2\|_{L^2_r}d\tau\\
                \lesssim_{a,T}&\|\phi\|_{L^2_r}+\int_0^t \|u_x(\tau)\|_{L^\infty_x} \|u(\tau)\|_{L^2_r}d\tau\\
                \lesssim_{a,T}&\|\phi\|_{L^2_r}+\int_0^t \|u(\tau)\|_{H^{s+\lambda}} \|u(\tau)\|_{L^2_r}d\tau\\
                \lesssim_{a,T}&\|\phi\|_{L^2_r}+\|\phi\|_{H^s}c(a,\lambda,T)\int_0^t \tau^{-\frac{\lambda}{1+a}} \|u(\tau)\|_{L^2_r}d\tau,
\end{split}
\end{equation}
where $3/2-s<\lambda<a+1$.
Thus, using the Gronwall's Lemma we obtain

\begin{equation}
\begin{split}
\|u(t)\|_{L^2_r}
                           \lesssim_{a,\lambda,T}& \|\phi\|_{L^2_r}.            
\end{split}
\end{equation}
From now, we can proceed as in the previous cases.

Case 4). $r=\theta$, $\theta \in (0,1/2]$. Similar to the case 3).

Case 5). $r=2+\theta \in (5/2,5/2+a)$, $\theta \in (1/2,1/2+a)$. 

For all $t\in [0,T]$, $0\leq \tau \leq t$, the inequality \eqref{well2} leads to
\begin{equation}\label{des20}
\begin{split}
\||x|^{2+\theta} U(t)\phi\|&=\|D_\xi^{2+\theta} (\psi(\xi,t)\ha)\|\\
                       &\lesssim_a \rho_3(T)\|\lanx^{2+\theta} \phi\|,
\end{split}
\end{equation}
and
\begin{equation}
\begin{split}\label{des21}
\||x|^{2+\theta} U(t-\tau)z(\tau)\|&=\|D_\xi^{2+\theta} (\psi(\xi,t-\tau)\hat{z}(\tau))\|\\
&\lesssim_a \rho_3(T)M\|\lanx^{2+\theta} u(\tau)\|.
\end{split}
\end{equation} 
The rest of the proof is similar to Case 2).

Case 6). $3/2+a <r \leq 5/2$. Follows from analogous to the case 3).

This completes the proof of Theorem \ref{well}.

\end{proof}

\begin{remark}
It is possible to give another proof of the well-posedness in $\z_{s,r}$ by using the Proposition 2.2 in \cite{FLP1}. See also \cite{dong}, for the more general version.
\end{remark}

\section{Proof of Theorem \ref{P1}}

The proof of the Theorems \ref{P1} and \ref{P2} in the case $a=1$ can be obtained by the same approach of \cite{pastran}. Thus, in the following we deal only with the case $a\in (0,1)$.

\begin{proof}[Proof of Theorem \ref{P1}]

Case $a\in (0,1/2)$.

The main idea of the proof is observe that the terms in \eqref{Derksi} have an appropriate decay when $|\xi|$ goes to infinity. First we consider the integral equation associated with IVP \eqref{bodiss}
\begin{equation}\label{121}
u(t)=U(t)\phi-\frac12\int_{0}^{t}U(t-\tau)u(\tau)\partial_{x}u(\tau)d\tau,
\end{equation}
where $$\widehat{U (t)\phi}(\xi)=\psi(\xi,t)\hat{\phi}(\xi).$$
Let $3/2+a=1+\gamma,$ where $\gamma \in (1/2+a,1)$, then multiplying \eqref{121} by $|x|^{1+\gamma}$ we obtain
\begin{equation}\label{aftH5}
D^{\gamma}_\xi\partial_\xi(\widehat{u(t)})=D^{\gamma}_\xi\partial_\xi\big(\psi(\xi,t)\hat{\phi}\big)
-\int_0^tD^{\gamma}_\xi\partial_\xi\big(\psi(\xi,t-\tau)\hat{z}\big)d\tau,
\end{equation}
where $z=\frac{1}{2}\partial_{x}u^{2}.$

Without loss of generality, we assume that $t_1=0<t_2$. Let $\phi \in \z_{3/2+a,3/2+a}$, then by the Theorem \ref{well}(i)) follows that
\begin{equation}
u\in C([0,T];H^{3/2+a}(\R)\cap L^2_r), \ \mbox{where} \ 0<r<3/2+a.
\end{equation}

Multiplying the linear part of \eqref{aftH5} by $\chi$ we obtain
\begin{equation*}
\begin{split}
\chi D_{\xi}^{\gamma}\partial_{\xi}(\psi(\xi,t)\hat{\phi})=&\ [\chi;D_{\xi}^{\gamma}]\partial_{\xi}(\psi(\xi,t) \hat{\phi})+D_\xi^{\gamma}(\chi \p_\xi (\psi(\xi,t) \hat{\phi}))\\
=&\ A+B,
\end{split}
\end{equation*}
where by \eqref{dksi} 
\begin{equation}\label{A}
\begin{split}
A&=[\chi;D_{\xi}^{\gamma}]\Big (t(1+a)\ps|\xi|^{a}\sgn(\xi) \hat{\phi}+2it\xi \ps\hat{\phi}+\ps\p_\xi \hat{\phi}\Big )\\
 &:=A_1+A_2+A_3,
\end{split}
\end{equation}
and
\begin{equation}\label{B}
\begin{split}
B&=t D_\xi^{\gamma} \Big(\chi \psi(\xi,t) \sgn(\xi)|\xi|^{a}\hat{\phi})+2itD_\xi^{\gamma}(\chi \psi(\xi,t) |\xi|\hat{\phi})+D_\xi^{\gamma}(\chi\psi(\xi,t) \p_\xi \hat{\phi}\Big)\\
 &:=B_1+B_2+B_3.
\end{split}
\end{equation}
The next result will be useful in our estimates.
\begin{claim}\label{claim1}
$ A_i,  B_j \in L^2$, for $i=1,2,3$ and $j=2,3$.
\end{claim}
\begin{proof}
By using Proposition \ref{Comu}, inequality \eqref{psi1} and Plancherel's identity, we obtain
\begin{equation}
\begin{split}
\|A_1\|&\lesssim \|\chi\|_{H^2_\xi}\|t(1+a)\ps|\xi|^{a}\sgn(\xi) \hat{\phi}\|\\
       &\lesssim_{t,a}  \|\phi\|,
\end{split}
\end{equation}

\begin{equation}
\begin{split}
\|A_2\|&\lesssim \|\chi\|_{H^2_\xi}\|2it\xi \ps\hat{\phi}\|\\
       &\lesssim_{t,a}  \|\phi\|,
\end{split}
\end{equation}
and
\begin{equation}
\begin{split}
\|A_3\|&\lesssim \|\chi\|_{H^2_\xi}\|\ps\p_\xi \hat{\phi}\|\\
       &\lesssim \|x\phi\|.
\end{split}
\end{equation}
For the $B_j$ terms, using \eqref{equiv} and \eqref{Leibh} 

\begin{equation}\label{B2}
\begin{split}
\|B_2\|\lesssim_t & \ \|\chi \psi(\xi,t) |\xi|\hat{\phi}\|+\|\D(\chi \psi(\xi,t) |\xi|\hat{\phi})\|\\
\lesssim_{t,a} & \  \|\ha\|+\|\D\ha\|\\
\lesssim_{t,a}& \|\phi\|+\||x|^{1/2+a}\phi\|,
\end{split}
\end{equation}
and
\begin{equation}\label{B3}
\begin{split}
\|B_3\|\leq & \ \| \chi \psi(\xi,t)\p_\xi\hat{\phi}\|+\|\D(\chi \psi(\xi,t)\p_\xi\hat{\phi})\|\\
       \lesssim_{t,a} & \ (\|\p_\xi \ha\|+\|\D\p_\xi \ha\|)\\
\lesssim_{t,a}& \|x\phi\|+\||x|^{3/2+a}\phi\|.
\end{split}
\end{equation}

\end{proof}


The integral part in \eqref{aftH5} can be written as
\begin{equation}\label{teoP1b}
\begin{split}
\int_{0}^{t}[\chi;&D_{\xi}^{\gamma}]\Big(\psi(\xi,t-\tau)\big((t-\tau)(1+a)|\xi|^a\mathrm{sgn}(\xi)\hat{z}+
2i(t-\tau)|\xi|\hat{z}+\p_\xi\hat{z}\\
&+D_{\xi}^{\gamma}\Big(\chi\big(\psi(\xi,t-\tau)\big((t-\tau)(1+a)|\xi|^a\mathrm{sgn}(\xi)\hat{z}+
2i(t-\tau)|\xi|\hat{z}+\p_\xi\hat{z}\big)\Big)d\tau\\
:=&\ \mathcal A_1+\mathcal A_2 +\mathcal A_3+\mathcal B_1+\mathcal B_2 +\mathcal B_3.
\end{split}
\end{equation}
\begin{claim}\label{claim2}
Let $t\in [0,T]$, then $\mathcal A_i, \mathcal B_i \in L^2$, for $i=1,2,3$.
\end{claim}
\begin{proof}
First of all, by an examining of proof of the Claim \ref{claim1} we see that was only used $\phi \in L^2(\langle x \rangle^{3/2+a}dx)$. Thus, we need to establish that the function $z$ also belongs to this space for every $t\in (0,T]$. For this, we observe that
\begin{equation}
\p_x(\lan^{3/2+a}u^2)=(3/2+a)\lan^{-1/2+a}2xu^2 + \lan^{3/2+a}\p_x u^2.
\end{equation}
Then it's enough estimate the left hand-side of the last identity. Therefore
\begin{equation}
\begin{split}
\|J(\lan^{3/2+a}u^2)\|&=\|J((\lan^{3/4+a/2}u)^2)\|\\
                      &\lesssim \|J(\lan^{3/4+a/2}u)\|^2\\
                      &\lesssim \|J^{\frac{1}{\beta}}u\|^{\beta}\|\lan u\|^{1-\beta},
\end{split}
\end{equation}
where above, we used the property $u\in C((0,T];H^{\infty})$ and the Proposition \ref{Jota} with $\beta=1/4-a/2$, $\nu=1$ and $\delta=\frac{1}{\beta}$.
For the $\mathcal B_1$ term, by \eqref{Leibh} and \eqref{equiv} 
\begin{equation}
\begin{split}
\|\mathcal B_1\|&\lesssim_{t,a} \int_0^t\|\D (\chi \psi(\xi,t-\tau))|\xi|^{a+1}\widehat{u^2})\|d\tau\\
                &\lesssim_{t,a} \|\widehat{u^2}\|+\|\D \widehat{u^2}\|\\
                &\lesssim_{t,a} \|u\|_{H^1}^2+\|u\|_{H^1}\||x|^{1/2+a}u\|.
\end{split}
\end{equation}
This finishes the proof of Claim \ref{claim2}.
\end{proof}

With respect to $B_1$ term, we can write
\begin{equation}\label{media0}
\begin{split}
B_1=&\ t(1+a)|\xi|^a \sgn(\xi)\psi(\xi)\chi(\xi)(\ha(\xi)-\ha(0))+t(1+a)|\xi|^a \sgn(\xi)\psi(\xi)\chi(\xi)\ha(0)  \\
      & :=B_{1,1}+B_{1,2}.
\end{split}
\end{equation}
Then $B_{1,1}\in H^1(\R)$, in fact
\begin{equation}\label{B11}
\begin{split}
\|B_{1,1}\|&\lesssim_{t,a}  \ \||\xi|^a \chi(\xi)\ha (\xi)\|+\||\xi|^a \chi(\xi)\ha(0)\|\\
           &\lesssim_{t,a}  \ \||\xi|^a \chi(\xi)\|_{\infty}\|\phi\|+\||\xi|^a \chi(\xi)\||\ha(0)|\\
           &\lesssim_{t,a}  \|\phi\|+|\ha(0)|,
\end{split}
\end{equation}
and
\begin{equation}\label{}
\begin{split}
\p_\xi B_{1,1}=& \ t(1+a)a|\xi|^{a}\chi \psi \frac{\ha(\xi)-\ha(0)}{|\xi|}+t(1+a)|\xi|^a \sgn(\xi)\p_\xi \psi \chi(\ha(\xi)-\ha(0))+\\
              &+t(1+a)|\xi|^a \p_\xi \chi \sgn(\xi)\psi(\ha(\xi)-\ha(0))+ t(1+a)|\xi|^a \sgn(\xi)\psi\chi\p_\xi \ha(\xi).
\end{split}
\end{equation}
Hence
\begin{equation}\label{b11}
\begin{split}
\|\p_\xi B_{1,1}\|\lesssim_{t,a}& \ \||\xi|^a \chi\|\|\p_\xi \ha\|_{\infty}+\||\xi|^a \chi \p_\xi \psi(\xi,t)\|_{L^{\infty}_\xi}\|\phi\|+\|\p_\xi \psi(\xi,t)\|_{L^\infty_\xi}\||\xi|^a \chi\||\ha(0)|\\
                 &+\||\xi|^a \p_\xi \chi\|_\infty \|\ha\|+\||\xi|^a \p_\xi \chi\| |\ha(0)|+\||\xi|^a \chi\|_\infty \|\p_\xi \ha\|\\
                 \lesssim_{t,a} & \ \|\langle x \rangle^{3/2+a}\phi\|+\|x\phi\|,
\end{split}
\end{equation}
where above we used the Sobolev embedding
$$\|\p_\xi \ha\|_\infty \lesssim \|J_\xi^{3/2+a}\ha\|=\|\lan^{3/2+a}\phi\|.$$
Therefore by Claim \ref{claim1}, Claim \ref{claim2} and \eqref{media0} -- \eqref{b11}, follows that $$t(1+a)|\xi|^a \sgn(\xi)\ps\chi(\xi)\ha(0)\in H^{\gamma}(\R).$$ Writing
\begin{equation}
\begin{split}
t(1+a)|\xi|^a \sgn(\xi)\psi\chi\ha(0)&=t(1+a)|\xi|^a \sgn(\xi)(\psi-1)\chi\ha(0)+t(1+a)|\xi|^a \sgn(\xi)\chi\ha(0)\\
                                           &:=C_1+C_2,
\end{split}
\end{equation}
follows that $C_1 \in H^1_\xi (\R)$. Thus $C_2 \in H^\gamma_\xi$, therefore from \eqref{equiv}

 $$t(1+a) \D (|\xi|^a \sgn(\xi)\chi\ha(0))\in L^2(\R).$$
 Since $\gamma=1/2+a$, by \eqref{Dstein1} it follows that $\ha(0)=0$. Therefore by \eqref{fourieru}, we obtain
 $$\hat{u}(0,t)=0,$$
 for every $t$ in which the solution exists.

 The case $a\in [1/2,1)$ follows by setting $3/2+a=2+\gamma$, where $\gamma=a-1/2$, using the derivative $\p_\xi^2(\psi(\xi,t)\ha(\xi))$ and noting that by \eqref{Dstein2}
$$\mathcal{D}_\xi^\gamma (|\xi|^{a-1}\chi(\xi))\notin L^2(\R).$$

 This finishes the proof of Theorem \ref{P1}.

\end{proof}

\section{Proof of Theorem \ref{P2}}

\begin{proof}[Proof of Theorem \ref{P2}]

Case $a\in (0,1/2)$. First of all, we assume without loss of generality that $0=t_1<t_2<t_3$. Since $\phi \in \z_{5/2+a,5/2+a}$, using the Theorem \ref{well}(ii)) we see that  
\begin{equation}
u\in C([0,T];H^{5/2+a}(\R)\cap L^2_r), \ \mbox{where} \ 0<r<5/2+a.
\end{equation}
Let $5/2+a=2+\gamma$, where $\gamma\in (0,1)$, then multiplying \eqref{121} by $|x|^{5/2+a}$ we obtain

\begin{equation}
\begin{split}
D_\xi^\gamma \p_\xi^2(\widehat{u(t)})&=D_\xi^\gamma \p_\xi^2(\psi(\xi,t)\ha)-\int_0^t (t-\tau)D_\xi^\gamma \p_\xi^2(\psi(\xi,t-\tau)\hat{z}(\tau))d\tau.\\
\end{split}
\end{equation}
Then by help of the $\chi$ function, we can write
\begin{equation*}
\begin{split}
\chi D_{\xi}^{\gamma}\partial_{\xi}^2(\psi(\xi,t)\hat{\phi})=&\ [\chi;D_{\xi}^{\gamma}]\partial_{\xi}^2(\psi(\xi,t) \hat{\phi})+D_\xi^{\gamma}(\chi \p_\xi^2 (\psi(\xi,t) \hat{\phi}))\\
:=&\ C+D,\\
\end{split}
\end{equation*}
where, using the second derivative in \eqref{Derksi2} we obtain
\begin{equation}
\begin{split}\label{C}
C=& \ [\chi;D_{\xi}^{\gamma}]\Bigg(\Big(\big(t^2 (1+a)^2|\xi|^{2a} -t(1+a)a|\xi|^{a-1}+4it^2 (1+a)|\xi|^{a+1}\sgn(\xi)\\
                         &- 4t^2 \xi^2 -2it \sgn(\xi) \big)\ha -\big( 2t(1+a)|\xi|^a \sgn(\xi)+4it|\xi|\big)\p_\xi \ha  +\p_\xi^2 \ha \Big)\psi(\xi,t)\Bigg)\\
                         :=& \ C_1+\ldots +C_8,
\end{split}
\end{equation}
and
\begin{equation}
\begin{split}\label{D}
D=&\ D_{\xi}^{\gamma}\Bigg(\chi\Big( \big(t^2 (1+a)^2|\xi|^{2a} -t(1+a)a|\xi|^{a-1}+4it^2 (1+a)|\xi|^{a+1}\sgn(\xi)\\
                         &- 4t^2 \xi^2 -2it \sgn(\xi) \big)\ha -\big( 2t(1+a)|\xi|^a \sgn(\xi)+4it|\xi|\big)\p_\xi \ha  +\p_\xi^2 \ha \Big)\psi(\xi,t)\Bigg)\\
                         :=& \ D_1+\ldots +D_8.
\end{split}
\end{equation}
We need of the following result.
\begin{claim}\label{claim3}
The above terms $C_{j},D_{j}\in L^2 (\R)$, where $j=1,...,8$, except $D_2$ and $D_6$.
\end{claim}
\begin{proof}
First of all, we deal with the $C_2$ term.

Since $\ha(0)=0$, by the Taylor Formula
\begin{equation}
\ha(\xi)=\xi \p_\xi \ha(0) +\int_0^\xi (\xi-\sigma)\p_\xi^2\ha(\sigma)d\sigma,
\end{equation}
and Proposition \ref{Comu} we obtain
\begin{equation}
\begin{split}
\|C_2\|&\lesssim \|\chi\|_{H^2_\xi}\|t(1+a)a|\xi|^{a-1}\psi \ha\|\\
             &\lesssim_a t\Big(\||\xi|^a \p_\xi \ha(0)\psi\|+\||\xi|^{a-1}\psi\int_0^\xi (\xi-\sigma)\p_\xi^2\ha(\sigma)d\sigma \|\Big)\\
             &\lesssim_a t\Big(\||\xi|^a e^{-t|\xi|^{1+a}}\||\p_\xi \ha(0)|+\||\xi|^{a}e^{-t|\xi|^{1+a}}\|\|\p_\xi^2\ha\|_{L^\infty_\sigma}\Big)\\
             &\lesssim_a t^{\frac{1}{2(1+a)}}\|\phi\|_{H^{5/2+a}},
\end{split}
\end{equation}
where above we used \eqref{lemasigma} and Sobolev embedding.
To deal with the terms $C_i$, $i\not= 2$, we can proceed similarly to terms $A_i$, $i=1,...,3$, in the Claim \ref{claim1}.

About the $D_j$ terms, using \eqref{equiv}, \eqref{Leibh} and \eqref{psi1}

\begin{equation}\label{D1}
\begin{split}
\|D_1\|&\lesssim  \ \| \chi t^2 (1+a)^2|\xi|^{2a}\psi(\xi,t)\ha\|+\|\D(\chi t^2 (1+a)^2|\xi|^{2a}\psi(\xi,t)\ha)\|\\
       &\lesssim_a  \ t^2(\|\ha\|+\|\D(|\xi|^{2a}\psi(\xi,t)\chi\ha)\|)\\ 
       &\lesssim_a  \ (t^2+t^{\frac{2}{1+a}}+t^{\frac{2+a}{1+a}}+t^{\frac{3}{1+a}})\|\phi\|+t^{\frac{2}{1+a}}\||x|^\gamma \phi\|\\
&\lesssim_{t,a} \|\phi\|+\||x|^{1/2+a}\phi\|,
\end{split}
\end{equation}

\begin{equation}\label{D3}
\begin{split}
\|D_3\|&\lesssim_a t^2  \Big( \| |\xi|^{a+1}\psi(\xi,t)\ha\|+\|\D(|\xi|^{a+1}\psi(\xi,t)\ha)\|\Big)\\
       &\lesssim_a  \ (t+t^{\frac{1+2a}{1+a}}+t^{\frac{2+a}{1+a}})\|\phi\|+t\||x|^\gamma \phi\|\\
&\lesssim_{t,a} \|\phi\|+\||x|^{1/2+a}\phi\|,
\end{split}
\end{equation}

\begin{equation}\label{D4}
\begin{split}
\|D_4\|&\lesssim_a t^2  \Big( \| \xi^2\psi(\xi,t)\ha\|+\|\D(\xi^2\psi(\xi,t)\ha)\|\Big)\\
       &\lesssim_a  \ (t^{\frac{2a}{1+a}}+t^{\frac{1+2a}{1+a}}+t^{\frac{3a}{1+a}})\|\phi\|+t^{\frac{2a}{1+a}}\||x|^\gamma \phi\|\\
&\lesssim_{t,a} \|\phi\|+\||x|^{1/2+a}\phi\|,
\end{split}
\end{equation}

\begin{equation}\label{D7}
\begin{split}
\|D_7\|&\lesssim_a \ t \Big( \||\xi|\psi(\xi,t)\p_\xi \ha\|+\|\D(|\xi|\psi(\xi,t)\p_\xi \ha)\|\Big)\\
       &\lesssim_a  \ (t^{\frac{a}{1+a}}+t+t^{\frac{2a}{1+a}})\|\p_\xi\ha\|+t^{\frac{a}{1+a}}\|\mathcal{D}_\xi^\gamma \p_\xi\phi\|\\
       &\lesssim_{t,a} \|x\phi\|+\||x|^{3/2+a}\phi\|,
\end{split}
\end{equation}

\begin{equation}\label{D8}
\begin{split}
\|D_8\|&\lesssim_a \Big( \|\psi(\xi,t)\p_\xi^2 \ha\|+\|\D(\psi(\xi,t)\p_\xi^2 \ha)\|\Big)\\
       &\lesssim_{t,a}  \|x^2 \phi\|+\|\D \p_\xi^2 \ha\|\\
       &\lesssim_{t,a} \|\phi\|+\||x|^{5/2+a}\phi\|,
\end{split}
\end{equation}

and finally

\begin{equation}\label{D5}
\begin{split}
\|D_5\|&\lesssim_a \ t \Big( \|\sgn(\xi)\psi(\xi,t)\ha\|+\|\D(\sgn(\xi)\psi(\xi,t)\ha)\|\Big).\\
\end{split}
\end{equation}
The last term can be estimated by similar way to \eqref{C5}.

This completes the proof of Claim \ref{claim3}.

\end{proof}

By way analogous to the linear part, we can write

\begin{equation}\label{intc}
\begin{split}
\int_{0}^{t}[\chi;&D_{\xi}^{\gamma}]\Big(\psi(t-\tau,\xi)\big((t-\tau)^2 (1+a)^2|\xi|^{2a} -(t-\tau)(1+a)a|\xi|^{a-1}+\\
&+4i(t-\tau)^2 (1+a)|\xi|^{a+1}\sgn(\xi)- 4(t-\tau)^2 \xi^2 -2i(t-\tau) \sgn(\xi) \big)\hat{z}\\
&-\big( 2(t-\tau)(1+a)|\xi|^a \sgn(\xi)+4i(t-\tau)|\xi|\big)\p_\xi \hat{z} +\p_\xi^2 \hat{z}\Big)\\
&+D_{\xi}^{\gamma}\Big(\chi\big(\psi(t-\tau,\xi)\big((t-\tau)^2 (1+a)^2|\xi|^{2a} -(t-\tau)(1+a)a|\xi|^{a-1}+\\
&4i(t-\tau)^2 (1+a)|\xi|^{a+1}\sgn(\xi))- 4(t-\tau)^2 \xi^2 -2i(t-\tau) \sgn(\xi) \big)\hat{z}-\\
&-\big( 2(t-\tau)(1+a)|\xi|^a \sgn(\xi)+4i(t-\tau)|\xi|\big)\p_\xi \hat{z} +\p_\xi^2 \hat{z}\Big)d\tau\\
:=&\ \mathcal C_1+\ldots+ \mathcal C_8+\mathcal D_1+\ldots +\mathcal D_8.
\end{split}
\end{equation}
\begin{claim}\label{claim4}
The terms $\mathcal C_j,\mathcal D_j \in L^2$, for all $t\in [0,T]$, $j=1,...,8$, except $\mathcal D_2$ and $\mathcal D_6$.
\end{claim}
\begin{proof} Similar to proof of Claim \ref{claim2}. In fact, looking at the proof of Claim \ref{claim3} we see that we only used that $\phi \in L^2(\lan^{5/2+a}dx)$. Then its enough to show that $z=uu_x \in L^2(\lan^{5/2+a}dx)$, for all $t\in (0,T]$.
We observe that
\begin{equation}
\p_x (\lan^{5/2+a}u^2)=(5/2+a)\lan^{3/2+a}2x u^2+\lan^{5/2+a}\p_x u^2.
\end{equation}
Therefore we will estimate the left-hand side of the last inequality
\begin{equation}
\begin{split}\label{ineq3}
\|J(\lan^{5/2+a}u^2)\|&=\|J((\lan^{5/4+a/2}u)^2)\|\\
                      &\lesssim \|J(\lan^{5/4+a/2}u)\|^2\\
                      &\lesssim \|J^{\frac{1}{\beta}}u\|^\beta \|\lan^{3/2}u\|^{1-\beta}.
\end{split}
\end{equation}
 In \eqref{ineq3} it was used $u\in C((0,T];H^\infty)$ and Lemma \ref{inter1} with $\nu=3/2$, $\delta=\frac{1}{\beta}$ and $\beta=\frac{1-2a}{4}$.

This finishes the proof of Claim \ref{claim4}.

\end{proof}




Since $\ha(0)=0$, using the Taylor's formula 

\begin{equation}
\begin{split}\label{taylor}
\ha(\xi)&=\xi \p_\xi \ha(0)+\int_0^\xi (\xi-\sigma)\p_\xi^2 \ha(\sigma)d\sigma,\\
\end{split}
\end{equation}
we obtain

\begin{equation}
\begin{split}
D_{2}=& \ t c_a \D \big(|\xi|^a \sgn(\xi)\psi \chi \p_\xi \ha(0)\big)+\\
&+ c_a\D \big(\underbrace{t|\xi|^{a-1} \sgn(\xi)\psi \chi \int_0^\xi (\xi-\sigma)\p_\xi^2 \ha(\sigma)d\sigma}_{F}\big)\\
:=& \ \bar{D}_2+R(\xi),
\end{split}
\end{equation}
hence $R\in H^1_\xi$, in fact

\begin{equation}
\begin{split}
\|F\|
&\lesssim_a  t\||\xi|^{a-1}\chi\xi^2\|\|\p_\xi^2 \ha\|_\infty\\
     &\lesssim_{a,t} \|J_\xi^{1/2+a} \p_\xi^2 \ha\|\\
     &\lesssim_{a,t} \|\lan^{5/2+a}\phi\|,
\end{split}
\end{equation}
and
\begin{equation}
\begin{split}
\|\p_\xi F\|\lesssim& \ t \Big(\||\xi|^{a-2}\chi \xi^2\|\|\p_\xi^2 \ha\|_\infty+\||\xi|^{a-1}\p_\xi \psi \xi^2 \chi\|\|\p_\xi^2 \ha\|_\infty+\\
           &+\||\xi|^{a-1}\psi \chi' \xi^2\|\|\p_\xi^2 \ha\|_\infty+\||\xi|^{a-1}\chi \int_0^\xi \p_\xi^2 \ha(\sigma)d\sigma\|\Big)\\
           \lesssim_{a,t}& \ \|\lan^{5/2+a}\phi\|.
\end{split}
\end{equation}
We can also write

\begin{equation}
\begin{split}
\bar{D}_2&=tc_a \D \Big(|\xi|^a \sgn(\xi)\chi \p_\xi \ha(0)(\psi-1+1)\Big)\\
          &=\bar{D}_{2,1}+\tilde{D}_2,
\end{split}
\end{equation}
where $\bar{D}_{2,1} \in L^2_\xi$.

In fact, since $$\p_\xi \psi(\xi,t)=-\Big(t(1+a)|\xi|^a \sgn(\xi)+2it|\xi|\Big)\psi,$$
we obtain
 
\begin{equation}
\begin{split}
\||\xi|^a \sgn(\xi)\chi \p_\xi \ha(0)(\psi-1)\|\lesssim_t\ |\p_\xi \ha(0)|\||\xi|^{a}\chi\|,\\
\end{split}
\end{equation}
and
\begin{equation}
\begin{split}
\|\p_\xi \big(|\xi|^a \sgn(\xi)\chi \p_\xi \ha(0)(\psi-1)\big)\|\lesssim_a & \  |\p_\xi \ha(0)|\Big( \||\xi|^a \chi \frac{\psi-1}{\xi}\|+\||\xi|^{a} \chi' \|+\\
&+\||\xi|^a \chi \p_\xi \psi\|\Big)\\
\lesssim_{a,t} & \ |\p_\xi \ha(0)|\big(\||\xi|^a \chi\|+\|(|\xi|^{2a}+|\xi|^{1+a})\chi\|+\\
&+\||\xi|^{a}\chi'\|\big).
\end{split}
\end{equation}

For the $D_6$ term we estimate as follows

\begin{equation}
\begin{split}
D_{6}&=-\D \Big (2t(1+a)|\xi|^a \sgn(\xi) \chi \p_\xi \ha (1+\psi-1)\Big)\\
        &:=D_{6,1}+D_{6,2},
\end{split}
\end{equation}
where $D_{6,2}\in L^2_\xi$, by the Lemma \ref{Leibnitz}.

Thus

\begin{equation}
\begin{split}
D_{6,1}&= -\D \Big(2t(1+a)|\xi|^a \sgn(\xi)\chi \p_\xi \ha(\xi)\Big)\\
          &=-2t(1+a)\D \Big(|\xi|^a \sgn(\xi)\chi \big(\p_\xi \ha(\xi)-\p_\xi \ha(0)\big)+|\xi|^a \sgn(\xi) \chi \p_\xi \ha(0)\Big)\\
          &:= \bar{D}_6+\frac{2}{a}\tilde{D}_2,
\end{split}
\end{equation}
where $\bar{D}_6 \in L^2_\xi$.

The last inequalities leads to 
\begin{equation}
\begin{split}\label{phi}
D_2+D_6=&\bar D_{2,1}+R+\bar D_6+D_{6,2}\\
&-(2+a)(1+a)t\mathcal{D}_\xi^{\gamma}(|\xi|^a \sgn(\xi) \chi \p_\xi \ha(0)).
\end{split}
\end{equation}
For the integral terms, by similar way we obtain
\begin{equation}
\begin{split}\label{int}
\mathcal D_2+\mathcal D_6=&\bar {\mathcal D}_{2,1}+\mathcal R+\bar {\mathcal D}_6+\mathcal D_{6,2}\\
&-(2+a)(1+a)\int_0^t (t-\tau)\mathcal{D}_\xi^{\gamma}(|\xi|^a \sgn(\xi)\chi \p_\xi \hat{z}(0,\tau))d\tau.
\end{split}
\end{equation}

From our hypotheses, Claim \eqref{claim3}, Claim \eqref{claim4}, \eqref{phi} and \eqref{int} follows that

\begin{equation}\label{equiv1}
D_\xi^{\gamma}\Bigg (|\xi|^{a}\sgn(\xi)\chi \Big( t\p_\xi \hat{\phi}(0)
 -\int_0^t (t-\tau)\p_\xi \hat{z}(0,\tau)\Big)d\tau \Bigg)\in L^2
\end{equation}
if and only if
\begin{equation}
D_{\xi}^{\gamma}\p_\xi^2 \hat{u}(\cdot,t)\in L^2 (\R).
\end{equation}

Also, from \eqref{consdata1}
\begin{equation}
\begin{split}\label{conser2}
\p_\xi \hat{z}(0,\tau)&=-i\widehat{xz}(0,\tau)\\
&=-\frac{i}{2}\int x \p_x u^2 (x,\tau)dx\\
&=\frac{i}{2}\|u(\tau)\|^2\\
                                                         &=i\frac{d}{d\tau}\int x u(x,\tau)dx.
\end{split}
\end{equation}

Using integrating by parts
\begin{equation}
\begin{split}\label{equiv2}
 t\p_\xi \ha(0) - \int_0^t (t-\tau)\partial_{\xi}\hat{z}(0,\tau)d\tau=& \  t\p_\xi \ha(0) - i\int_{0}^{t}(t-\tau)\frac{d}{d\tau}\int x u(x,\tau)dxd\tau\\
=& -i t \int x\phi(x)dx-i(t-\tau)\int xu(x,\tau)dx\Big |_{\tau=0}^{\tau=t} -\\
&- i\int_{0}^{t}\int x u(x,\tau)dxd\tau\\
=&\ -i t \int x\phi(x)dx +i t \int x\phi(x)dx-\\
&-i\int_{0}^{t}\int x u(x,\tau)dxd\tau\\
&=-i\int_{0}^{t}\int x u(x,\tau)dxd\tau.
\end{split}
\end{equation}
Putting $t=t_2$, \eqref{equiv1} and \eqref{equiv2} implies that 

\begin{equation}
\begin{split}
D_\xi^{\gamma}(|\xi|^{a}\sgn(\xi)\chi)\int_{0}^{t_2}\int x u(x,\tau)dxd\tau \in L^2,
\end{split}
\end{equation}
 by Stein derivative
\begin{equation}
\begin{split}
\mathcal{D}_\xi^{\gamma}(|\xi|^{a}\sgn(\xi)\chi)\int_{0}^{t_2}\int x u(x,\tau)dxd\tau \in L^2.
\end{split}
\end{equation}
Hence, in view of $\gamma=1/2+a$, from \ref{Dstein1} we see that
\begin{equation}
\int_{0}^{t_2}\int x u(x,\tau)dxd\tau=0.
\end{equation}
 By  Rolle's lemma, there exists $\tau_1\in (0,t_2)$ such that
\begin{equation}\label{rolle1}
\int xu(x,\tau_1)dx=0.
\end{equation}
Analogously, using that \ $u(t_2),u(t_3)\in \z_{5/2+a,5/2+a}$ \ we can show that
exists $\tau_2 \in (t_2,t_3)$ such that
\begin{equation}\label{rolle2}
\int xu(x,\tau_2)dx=0.
\end{equation}
Finally from \eqref{rolle1}, \eqref{rolle2} and identity \eqref{consdata1} we obtain $u(t)=0$, for all $t \in [\tau_1,\tau_2]$. In view of $\|u(t)\|$ is decreasing in $t$, we conclude $$u(t)=0, \ \mbox{\for all} \ t\geq \tau_1=\bar{t}.$$

The case $a\in [1/2,1)$ can be deal by choosing $5/2+a=3+\gamma$, where $\gamma=a-1/2$, using the derivative $\p_\xi^3(\psi(\xi,t)\ha(\xi))$ and observing that by \eqref{Dstein2}
$$\mathcal{D}_\xi^\gamma (|\xi|^{a-1}\chi(\xi))\notin L^2(\R).$$
\end{proof}

\bibliographystyle{mrl}

\end{document}